\documentclass[12pt, a4paper, leqno]{amsart}

\usepackage[utf8]{inputenc}
\usepackage{amsmath} 
\usepackage{amsfonts}
\usepackage{amssymb}
\usepackage{txfonts}   
\usepackage{amsfonts}  
\usepackage{comment}
\usepackage{ae}   
\usepackage{graphicx}   
\usepackage{pst-all} 
\usepackage{xcolor} 
 
\newcommand{\R}{\varmathbb{R}}
\newcommand{\Z}{\varmathbb{Z}}
\newcommand{\Rn}{{\varmathbb{R}^n}}
\newcommand{\Rk}{\varmathbb{R}^2}

\newcommand{\vint}{\operatornamewithlimits{\boldsymbol{--} \!\!\!\!\!\! \int }}

\def\diam{\qopname\relax o{diam}}

\def\dist{\qopname\relax o{dist}}
\def\min{\qopname\relax o{min}}
\def\diam{\qopname\relax o{diam}}

\def\Cig{\qopname\relax o{Cig}}
\def\phi{\varphi}

\let\oldmarginpar\marginpar
\renewcommand\marginpar[1]{\-\oldmarginpar[\raggedleft\footnotesize #1]%
{\raggedright\footnotesize #1}}

\usepackage{amsthm}

\swapnumbers
\theoremstyle{plain}
\newtheorem{theorem}[equation]{Theorem}
\newtheorem{lemma}[equation]{Lemma}

\newtheorem{corollary}[equation]{Corollary}

\theoremstyle{definition}
\newtheorem{definition}[equation]{Definition}

\newtheorem{example}[equation]{Example}

\theoremstyle{remark}
\newtheorem{remark}[equation]{Remark}

\numberwithin{equation}{section}

\title{Embeddings  into Orlicz spaces via the modified Riesz potential}

\author{Petteri Harjulehto}
\address[Petteri Harjulehto]{Department of Mathematics and Statistics,
FI-20014 University of Turku, Finland}
\email{petteri.harjulehto@utu.fi}

\author{Ritva Hurri-Syrj\"anen}
\address[Ritva Hurri-Syrj\"anen]{Department of Mathematics and Statistics,
FI-00014 University of Helsinki, Finland}
\email{ritva.hurri-syrjanen@helsinki.fi}

\date{\today}

\begin{document}


\begin{abstract} $L^1_1$-functions which are defined in non-smooth  domains 
in the $n$-dimensional Euclidean space can be estimated point-wise by the modified Riesz potential of their gradients. These point-wise estimates imply embeddings into Orlicz spaces from the space $L^1_p$, $1\le p <n$, where the functions are defined in bounded or unbounded domains with minimum requirement of the smoothness  of the boundary. The results are sharp for  $L^1_1$-functions. 
\end{abstract}

\maketitle

\markboth{\textsc{Petteri Harjulehto and Ritva Hurri-Syrj\"anen }}
{\textsc{Petteri Harjulehto and Ritva Hurri-Syrj\"anen }}

%%%%%%%%%%%%%%%%%%%%%%%%%%%%%%%%%%%%%%%%%%%%%%%%%%%%%%%
\section{Introduction}

It is well known that 
a locally Lipschitz function
can be estimated point-wise 
by the Riesz potential of its gradient
in bounded John domains,
\cite[Theorem]{R}, \cite[Theorem 10]{Haj01},
and hence, especially, in Lipschitz domains and
in convex domains,
\cite[Lemma 7.16]{Gilbarg-Trudinger}.
By modifying the Riesz potential,
point-wise
estimates  
can be generalized for 
functions which are defined in more irregular domains than John domains, 
\cite[Theorem 3.4]{HH-SK}, \cite[Theorem 4.4]{HH-S1}.
More precisely, for every function $u$ whose weak distributional partial derivatives are in $L^1(G)$,
the pointwise estimate
\begin{equation}\label{modified}
|u(x) - u_D| \le \int_G \frac{|\nabla u(y)|}{\psi(|x-y|)^{n-1}} \, dy
\end{equation}
holds
for almost every $x \in G$. 
Here, $G$ is a domain in the $n$-dimensional Euclidean space and the regularity of the boundary is controlled by the function $\psi$.
Hedberg's method
\cite[Lemma, Theorem 1]{Hed72}
can be extended so that this point-wise estimate leads to the Sobolev-type inequality where  an Orlicz-space is the target space.
Hedberg's method has been used by A. Cianchi and B. Stroffolini for the classical Riesz potential when
functions are Orlicz functions,
\cite[Theorem 1, Corollary 1]{CS}, and 
by the authors for the modified Riesz potential
with a special Orlicz function,
\cite[Theorem 1.1]{HH-SK} and \cite[Theorem 1.1]{HH-S},
and with a general Orlicz function in
\cite[Corollary 3.4, Corollary 5.4]{HH-S1}.
For other papers on Orlicz embeddings of Cianchi we refer to \cite{Cianchi1996}, \cite{Cianchi1999}.

In the present paper we show that the optimal Orlicz function for the modified Riesz potential 
in \eqref{modified} can be found as a function of $\psi$
which depends on the geometry of the domain $G$.
 Our main theorem is the following theorem where we give the formula to the Orlicz function.

\begin{theorem}\label{thm:defn_H_intro}
Let $1\le p<n$.
Let the continuous, strictly increasing function $\varphi : [0,\infty )\to [0,\infty )$  be such that 
$\phi(0) = \lim_{t \to 0^+} \phi(t) =0$ and
$\phi$ satisfies
the $\Delta_2$-condition and
the inequality $\frac{\varphi (t_1)}{t_1}\le
\frac{\varphi (t_2)}{t_2}$ whenever   $0<t_1\le t_2$.
If
\begin{equation}\label{equ:psi_a}
\psi (t) = \begin{cases}
\phi(t) & \quad \text{when} \quad 0\le t \le 1;\\
\phi(1) t & \quad \text{when} \quad t \ge 1,\\
\end{cases}
\end{equation}
then there exists an $N$-function $H$  that satisfies the $\Delta_2$-condition, and  
\[
H^{-1}(t)  \approx \frac{t^{\frac1p-1}}{\psi \left(t^{- \frac{1}{n}} \right)^{n-1}}
\quad \text{for} \quad t>0.
\]
\end{theorem}

With this function we obtain the following point-wise estimate.

\begin{theorem}\label{thm_main}
Let $G$ be a domain in $\Rn$, $n\geq 2$.
Let $1\le p <n$.
If $H$ is the function from Theorem  \ref{thm:defn_H_intro}
and $\| f \|_{L^p (G)} \le 1$, then 
there exists  a constant $C$ such that 
the point-wise estimate
\begin{equation*}%\label{main_riesz_inequality}
H\left ( \int_{G}  \frac{|f(y)|}{\psi(  |x-y|)^{n-1}} \, dy\right )
\le C (M f (x))^p
\end{equation*}
holds for every $x\in \Rn$. Here,  $M f$ is the Hardy-Littelewood maximal operator of $f$ and the constant $C$ depends on $n$, $p$, and the $\Delta_2$-constant of $H$ only.
\end{theorem}

By this point-wise estimate
we obtain embedding results for bounded and unbounded 
non-smooth domains.
Examples of these domains are Lipschitz domains and convex domains, but also 
domains with suitable outward cusps are allowed.

We define a class of domains which are controlled by the function $\psi$ 
from \eqref{equ:psi_a}. We call these domains  in Definition \ref{john} as
$\phi$-cigar John domains, since our definition is a modification of \cite[2.1]{Vaisala} where 
J. V\"ais\"al\"a has defined unbounded John domains with $\phi (t)=t$.
Hence, examples of  $\phi$-John domains  are the classical bounded and unbounded John domains, but also so called
$s$-John domains when $\phi(t)=t^s$.

We have the following corollary which recovers some of the known results of the Poincar\'e inequality.

\begin{corollary}\label{corollary_bdd}
If  there exists $\alpha \in[1, n/(n-1))$ such that
$t^\alpha/\phi(t)$ is increasing for $t>0$ and  if $D$  is  a bounded or an unbounded $\phi$-cigar John domain with a constant $c_J$  in $\Rn$, $n \ge 2$
and if $1\le p<n$, then with the function $H$ in Theorem \ref{thm:defn_H_intro} there exists a constant $C$ such that the inequality
\begin{equation*}
\inf_{b\in\R}
\| u-  b \|_{L^H(D)} \le C \|\nabla u \|_{L^p (D)},
\end{equation*}
holds 
for every $u \in L^1_{\textup{loc}}(D)$ with $|\nabla u| \in L^p(D)$. Here the constant $C$ depends  on $n$, $p$, $\Delta_2$-constants of $H$ and $\phi$, and John constant $c_J$ only.
\end{corollary}

We point out that
if $D$ is a bounded $s$-John domain, then $\phi(t) = t^s$, $t\geq 0$,  and this corollary yields that  the $\left(\frac{np}{n-np+sp(n-1)}, p \right)$-Poincar\'e inequality holds.
If $p=1$, the result is optimal. Thus the corollary  recovers some of the known results of  \cite[Theorem 10]{SmiS90},  \cite[Corollaries 5 and 6]{HajK98}, and \cite[Theorem 2.3]{KilM00}, but our proof is completely different from the previous proofs.

Especially, in Section 6 we construct an example of an unbounded domain which shows  that the Lebesque space cannot be the target space in this corresponding embedding if $\lim_{t \to 0^+} t/\phi(t) =\infty$.

The outline of the paper is as following:
We define  the domains we consider in Section 2 and we call them $\phi$-cigar John domains. We find the suitable Orlicz function  in Section 3,
prove embedding theorems in Section 4, recover some  Poincar\'e inequalities in Section 5, and
in Section 6 we construct an example of an unbounded $\phi$-cigar domain.

%%%%%%%%%%%%%%%%%%%%%%%%%%%%%%%%%%%%%%%%%%%%%%%%%%%%%%%%%%%%%%%%%%%%%%%%%%%%%%%%%%%
\section{John domains}
\label{John}

Throughout the paper we
let the function $\varphi : [0,\infty )\to [0,\infty )$ satisfy the following conditions
\begin{enumerate}
\item[(1)] $\phi$ is continuous,
\item[(2)] $\phi$ is strictly increasing,
\item[(3)] $\phi(0) = \lim_{t \to 0^+} \phi(t) =0$,
\item[(4)] there exists a constant $C_\varphi \ge 1$ such that 
\[
\frac{\varphi (t_1)}{t_1}\le C_\varphi
\frac{\varphi (t_2)}{t_2}
\]
 whenever   $0<t_1\le t_2$,
\item[($5$)] $\phi$ satisfies the $\Delta_2$-condition i.e.\ there exists a constant $C_\phi^{\Delta_2} \ge 1$ such that $\phi(2t) \le C_\phi^{\Delta_2} \phi(t)$ for every $t>0$.
\end{enumerate}
We write
\begin{equation}\label{equ:psi}
\psi (t) = \begin{cases}
\phi(t) & \quad \text{if} \quad 0\le t \le 1;\\
\phi(1) t & \quad \text{if} \quad t \ge 1.\\
\end{cases}
\end{equation}
Now, if $\phi$ satisfies the conditions (1)--(5), then $\psi$ does, too, and the constant in (4) is the same for the functions $\phi$ and $\psi$, that is 
$C_{\phi}=C_{\psi}$.

The definition of a bounded John domain goes back to F. John
\cite[Definition, p. 402]{J} who defined an inner radius and an outer radius domain, and later this domain was renamed as a John domain
in  \cite[2.1]{MS79}.

We extend the definition of John domains following J.~V\"ais\"al\"a
\cite[2.1]{Vaisala} in the classical case.
Let $E$ in $\Rn$, $n\geq 2$, be a closed rectifiable curve with endpoints $a$ and $b$. The subcurve between $x\,,y \in E$ is denoted by $E[x,y]$. For $x \in E$ we write
\[
q(x) = \min\bigg\{\ell\Big(E[a,x] \Big), \ell\Big(E[x,b] \Big) \bigg\},
\]
where $\ell\big(E[a,x]\big)$ is the length of the subcurve $E[a,x]$.

\begin{definition}\label{john}
A bounded or an unbounded domain $D$ in $\Rn$ is a $\varphi$-cigar John domain if there exists a constant $c_J >0$ 
such that each pair of points $a, b \in D$ can be joined by a closed rectifiable curve $E$  in $D$ such that
\[
\Cig E(a,b) = \bigcup\left\{B \left(x,  \frac{\psi(q(x))}{c_J} \right): x \in E\setminus\{a,b\} \right\} \subset D
\] 
where $B(x,r)$ is an open ball centered at $x$ with a radius $r>0$ and the function 
$\psi$ is defined as in \eqref{equ:psi}.
\end{definition}

The set
$\Cig E(a,b)$ is called a cigar with core $E$ joining $a$ and $b$. 
We point out that if $D$ is a $\varphi$-cigar John domain with $\varphi(t) = t^p$, $p \ge 1$, then it is a $\varphi$-cigar John domain with $\varphi(t) = t^q$ for every  $q \ge p$.
For the case $\psi (t)=\varphi (t)=t$ for all $t\geq 0$,
in Definition \ref{john}, we refer to \cite[2.1]{Vaisala} and \cite[2.11 and 2.13]{Nakki_Vaisala}.

If $D$ is a bounded domain then the following definition from \cite[Definition 4.1]{HH-S1}
for a $\psi$-John domain
gives an equivalent definition
to a bounded $\varphi$-cigar John domain.

\begin{definition}\label{bounded-john}
A bounded domain $D$ in $\Rn\,, n\geq 2\,,$ is a $\psi$-John domain if there exist a constants $0< \alpha \le \beta<\infty$ and a point $x_0 \in D$ such that each point $x\in D$ can be joined to $x_0$ by a rectifiable curve $\gamma:[0,\ell(\gamma)] \to D$, parametrized by its arc length, such that $\gamma(0) = x$, $\gamma(\ell(\gamma)) = x_0$, $\ell(\gamma)\leq \beta\,,$ and
\[
\psi(t) \leq  \frac{\alpha}{\ell(\gamma)} \dist\big(\gamma(t), \partial D\big) \quad \text{for all} \quad t\in[0, \ell(\gamma)].
\]
The point $x_0$ is called a John center of $D$ and 
$\gamma$ is called a John curve of $x$.
\end{definition}

If  the function $\psi$ is defined as in \eqref{equ:psi} with the function $\phi$, then
a bounded domain is a $\psi$-John domain if and only if it is a $\phi$-John domain.  If $\psi(t) =t$, then our definition for bounded $\psi$-John domains coincides with the  definition of the classical John domains. If $\psi(t) = t^s$, $s\ge 1$, then our definition for bounded $\psi$-John domains coincides with the  definition of $s$-John domains.

\begin{theorem}\label{thm:John-John}
Let $D$ be a bounded domain. 
If $D$ is a $\psi$-John domain then $D$ is a $\varphi$-cigar John domain.
On the other hand, if $D$ is a $\varphi$-cigar John domain with a constant $c_J$, then $D$ is a $\psi$-John domain with constants
\[
\alpha= \frac{c_J\, \phi(1) \left(\max\left\{2, \frac{c_J \diam(D)}{\phi(1)} \right\} \right)^2}{\psi \left( \frac1{2c_J} \psi \left(\frac14 \diam(D)\right)\right)},
\]
\[
\quad
\beta = \max\left\{2, \frac{c_J \diam(D)}{\phi(1)} \right\}.
\]
\end{theorem}

Note that when $\diam(D) \to \infty$, then $\alpha \to \infty$ with the same speed as $\diam(D)$.

\begin{proof}
Assume first that $D$ is a $\psi$-John domain with a John center $x_0$. Let $a, b \in D$ and let the John curves $\gamma_1$  and $\gamma_2$ connect them to $x_0$,
respectively. We may assume that $a,b \in D \setminus B(x_0, \dist(x_0, \partial D))$, since inside the ball the points can be connect by two straight lines going via the center of the ball $B(x_0, \dist(x_0, \partial D))$.
Let $E =  \gamma_1 \circ \gamma_2$.  Then,  
\[
\begin{split}
&\Cig E(a, b)\\ &= \bigcup_{t \in (0, \ell(\gamma_1)]} B \left( \gamma_1(t),  \frac{\psi(t)}{\alpha/\dist(x_0, \partial D)} \right) \cup 
\bigcup_{t \in (0, \ell(\gamma_2)]} B \left( \gamma_2(t),  \frac{\psi(t)}{ \alpha/\dist(x_0, \partial D)} \right)\\
%&\subset \bigcup_{t \in (0, \ell(\gamma_1)]} B \left( \gamma_1(t),  \frac{\psi(t)}{\alpha/\ell(\gamma_1)} \right) \cup 
%\bigcup_{t \in (0, \ell(\gamma_2)]} B \left( \gamma_2(t),  \frac{\psi(t)}{ \alpha/\ell(\gamma_2)} \right) \subset D
\end{split}
\]
and thus $D$ is a $\varphi$-cigar John domain.

Assume then that $D$ is a $\varphi$-cigar John domain. 
Let us carefully choose  a suitable John center so that
the center is not too close to the boundary of $D$. Let $x, y \in D$ such that $|x-y|\ge \frac12 \diam(D)$. Let $E$ be a core of a John cigar that connects$x$ and $y$.
Then the length of $E$ is at least $\frac12 \diam(D)$. Let $x_0$ be the center of $E$.  Then
\[
\dist(x_0, \partial D) \ge \frac{\psi(\frac14 \diam(D))}{c_J}
\] 
so we choose 
$r= \psi \Big(\frac14 \diam(D) \Big)/c_J$, and hence $B(x_0, r) \subset D$.
From now on this $r$ and the point $x_0$ are fixed in this proof.
%$r= \frac{\psi(\frac12 \diam(D))}{c_J}$.

For every $a \in D \setminus\overline{ B(x_0, r)}$ there exists a curve $E$ such that  $\Cig E(a, x_0) \subset D$.
Let $\ell (E)$ be the length of $E$, then $\ell (E) \le 2$ or by the definition
\[
\diam(D) \ge 2 \frac{\psi(\ell(E)/2)}{c_J}= 2 \frac{\phi(1) \ell (E)}{2 c_J}
\]
i.e. $\ell (E) \le \max\left\{2, \frac{c_J \diam(D)}{\phi(1)} \right\}= \beta$.

\begin{figure}[ht!]
\includegraphics[width=11 cm]{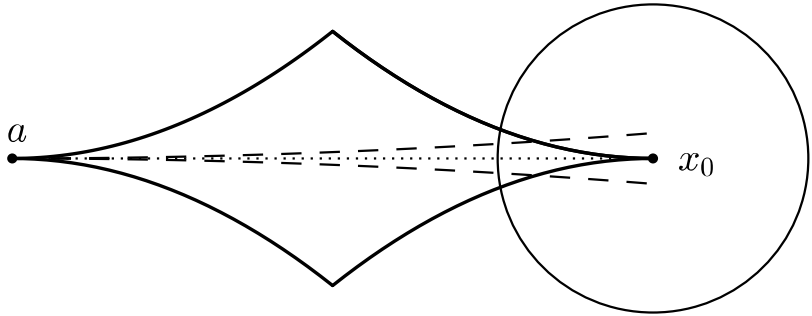}
\caption{The cigar from $a$ to $x_0$ (the solid line), the core $E$ (the dotted line) and a new carrot  given by the constant $c_J M$ (the dashed line).}
\end{figure}

Note that the length of $E$ inside the ball $B(x_0, r)$ is at least $r$ and thus for the points in $E \cap \partial B(x_0, r)$ the distance to the boundary is  at least  $\psi(r/2)$.
Let us choose that
 \[
M = \frac{\psi(\beta)}{\psi(r/2)}= \frac{\phi(1)\beta}{\psi(r/2)}.
 \]
Since $r \le \ell(E) \le \beta$ and $\psi$ is increasing, we have  $M \ge 1$.

Let $z_0 \in E$ be the first point from $a$ that satisfies $z_0 \in \partial B(x_0, r)$. Let us replace $E[z_0, x_0]$ by the radius of the ball $B(x_0, r)$, if necessary. Let us denote this new arc by $E$.
Let $\gamma$ be an arc $E$ parametrized by its curve length, such that $\gamma(0) = a$, $\gamma(\ell(E) ) = x_0$.  Since
\[
\frac{\psi(\ell(E))}{M c_J}  \le \frac{\psi(r/2)}{c_J}
\]
we obtain that
\[
\bigcup_{t \in(0, \ell(E) )}B \left(\gamma(t), \frac{\psi(t)}{M c_J} \right) \setminus B(x_0, r) \subset \Cig[a, x_0].
\]
This yields that
\[
\bigcup_{t \in(0, \ell(E))}B \left(\gamma(t), \frac{\psi(t)}{M c_J} \right)  \subset D
\]
and thus 
\[
\psi(t) \le M c_J \dist(\gamma(t), \partial D)
\le \frac{M c_J \beta}{\ell(E)} \dist(\gamma(t), \partial D).
\]
This yields that  we may choose $\alpha = M c_J \beta$. Thus, $D$ is a $\psi$-John domain with these $\alpha$ and $\beta$.
\end{proof}

%%%%%%%%%%%%%%%%%%%%%%%%%%%%%%%%%%%%%%%%%%%%%%%%%%%%%%%%%%%%%%%%%%%%%%%%%%%%%%%%%%%%%
%%%%%%%%%%%%%%%%%%%%%%%%%%%%%%%%%%%%%%%%%%%%%%%%%%%%%%%%%%%%%%%%%%%%%%%%%%%%%%%%%%%%%
%%%%%%%%%%%%%%%%%%%%%%%%%%%%%%%%%%%%%%%%%%%%%%%%%%%%%%%%%%%%%%%%%%%%%%%%%%%%%%%%%%%%%
\section{Point-wise estimates}

We note that  by the condition (4) of $\phi$ 
\begin{equation}\label{varphi_bdd}
\psi (t)\le  C_\phi \phi(1)  t  \quad \text{for all}   \quad  t  \ge 0.
\end{equation}

We recall a covering lemma from \cite[4.3. Lemma]{HH-S1} which is valid 
for a bounded $\varphi$-John domain.  For the previous versions in classical case we refer to \cite[Theorem 9.3]{Hajlasz-Koskela} and in a special case to \cite[Lemma 3.5]{HH-SK}.

\begin{lemma}\cite[4.3. Lemma]{HH-S1}.\label{lem:covering} 
Let $\phi$ satisfies the conditions (1)--(5).
Let $\psi  :[0,\infty )\to [0,\infty )$ be defined as in \eqref{equ:psi}.
 Let $D$ in $\Rn\,, n\geq 2\,,$ be a bounded $\psi$-John domain 
with  John constants $\alpha$ and $\beta$. Let  $x_0 \in D$ the John center. 
Then for every $x\in D\setminus B(x_0, \dist(x_0, \partial D))$ there exists a sequence of balls $\big(B(x_i, r_i)\big)$ such that $B(x_i, 2r_i)$ is in $D\,$ for each
$i=0,1,\dots\,,$ and
for some constants $K=K(\alpha, \dist(x_0, \partial D), \diam (D),\phi )$, $N=N(n)$, and $M=M(n)$ 
\begin{itemize}
\item
$B_0 = B\Big(x_0, \frac12 \dist(x_0,  \partial D)\Big)$;
\item
$\psi(\dist(x, B_i))\leq K r_i$, and $r_i \to 0 $ as $i\to \infty$;
\item
no point of the domain $D$ belongs to more than $N$ balls $B(x_i, r_i)$; and
\item
$|B(x_i, r_i) \cup B(x_{i+1}, r_{i+1})| \leq M |B(x_i, r_i) \cap B(x_{i+1}, r_{i+1})|$.
\end{itemize}
\end{lemma}

\begin{remark} (1)
The constant $K$ in the previous lemma can be taken to be
$K=\max\{2\alpha/\dist (x_0,\partial D), 2\phi(1), \phi(\diam (D))/\diam (D)\}.$\\
(2) If $D$ is a $\phi$-cigar John domain and the John center has been chosen as in  Theorem~\ref{thm:John-John}, then 
\[
\frac{\alpha}{\dist(x_0, \partial D)}\le \frac{c_J^2\, \phi(1) \left(\max\left\{2, \frac{c_J \diam(D)}{\phi(1)} \right\} \right)^2}{\psi \left( \frac1{2c_J} \psi \left(\frac14 \diam(D)\right)\right) \psi\left(\frac14 \diam(D)\right)}
\to \frac{32 c_J^5}{\phi(1)^4}
\]
as $\diam(D) \to \infty$.
\end{remark}

We recall the following definitions.
Let $G$ be an open set of $\Rn$.
We denote the Lebegue space by $L^{p}(G)$, $1\le p < \infty$. By $L^1_p(G)$, $1\le p < \infty$, we denote those locally integrable functions whose first weak distributional derivatives belongs to $L^p(G)$ i.e.\  $L^1_p(G) =\left \{ u \in L^1_{\textup{loc}}(G): |\nabla u| \in L^p(G) \right\}$.
By $W^{1,p}(G)$, $1\le p < \infty$, we denote those functions from $L^p(G)$ whose first weak distributional derivatives belongs to $L^p(G)$ i.e.\  $W^{1,p}(G) =\left \{ u \in L^p(G): |\nabla u| \in L^p(G) \right\}$.

Theorem~\ref{thm:John-John} and  Lemma~\ref{lem:covering}  give the following point-wise estimate which we recall from 
\cite[4.4. Theorem]{HH-S1}.

\begin{theorem}\label{thm:Riesz}
Let $\phi$ satisfy the conditions (1)--(5).
Let $\psi  :[0,\infty )\to [0,\infty )$ be as defined in \eqref{equ:psi}.
 Let $D$ in $\Rn\,, n\geq 2\,,$ be a bounded $\phi$-cigar John domain 
with a John constant $c_J$ .
Then there exists a finite constant $C$ and $x_0 \in D$ such that for every $u\in L^1_1(D)$ and for almost every $x\in D$ 
the inequality
\[
\big|u(x) - u_{B(x_0,  \dist(x_0, \partial D))}\big| \leq C \int_{D}  \frac{|\nabla u(y)|}{\psi \big( |x-y|\big)^{n-1}} \, dy
\]
holds.
Here
\begin{equation*}
C = c\left(n, c_J, C_\phi, C_\phi^{\Delta_2}, \phi(1), \min\bigg\{\diam(D), 1\bigg\}\right).
\end{equation*}
\end{theorem}

We recall the definitions of $N$-functions and Orlicz spaces.

\begin{definition}
A  function $H:[0, \infty) \to [0, \infty)$ is an $N$-function if
\begin{enumerate}
\item[(N1)] $H$ is continuous,
\item[(N2)] $H$ is convex,
\item[(N3)] $\lim_{t \to 0^+}\frac{H(t)}{t} =0$ and $\lim_{t \to \infty}\frac{H(t)}{t} =\infty$.
\end{enumerate}
\end{definition}

Continuity and $\lim_{t \to 0^+}\frac{H(t)}{t} =0$ yield that $H(0)=0$. Let $0<t<s$ by convexity
\[
H(t) = H\left(\frac{t}{s} s + \left(1- \frac{t}{s} \right)0 \right)
\le \frac{t}{s} H(s) + \left(1- \frac{t}{s} \right) H(0)
\]
and thus $\frac{H(t)}{t} \le \frac{H(s)}{s}$ for $0<t<s$.

This implies that $H$ is a strictly increasing function.

By the notation $f\lesssim g$ we mean that there exists a constant 
$C>0$ such that $f(x)\le C g(x)$ for all $x$. The notation $f\approx g$ means that 
$f\lesssim g\lesssim f$.

Two $N$-functions $H$ and $K$ are equivalent, which is written as
$H\simeq K$, if there exists $m\ge 1$ such that 
$H(t/m)\le K(t)\le H(mt)$ for all $t>0$.
Equivalent $N$-functions give the same space with 
comparable norms. 
We point out that $H\simeq K$ if and only if for the inverse functions $H^{-1}\approx K^{-1}$.

We  assume that $H$ satisfies the $\Delta_2$-condition, that is, there exists a constant
$C^{\Delta_2}_H$ such that
\begin{equation}\label{H_doubling}
H (2t)\le C^{\Delta_2}_H H(t) \quad \text{for all} \quad t>0.
\end{equation}
If an $N$-function satisfies the $\Delta_2$-condition then the relations
$\simeq$ and $\approx$ are equivalent. 
The constant $C^{\Delta_2}_H$ is called the 
$\Delta_2$-constant of $H$.

Let $G$ in $\Rn$ be an open set.
The Orlicz class is a set of all measurable functions
$u$
defined on $G$
such that
\[
\int_G H \Big(|u(x)| \Big) \, dx < \infty\,.
\]
We study the Orlicz space $L^H (G)$ which means 
the space of all measurable functions
$u$ defined on $G$ such that 
\[
\int_G H \Big(\lambda |u(x)| \Big) \, dx < \infty
\]
for some  $\lambda >0$.

Whenever the function $H$
satisfies the $\Delta_2$-condition, then the space $L^H (G)$ is a vector space and it is 
equivalent to the corresponding Orlicz class.
We study these Orlicz spaces and call their functions Orlicz functions.
The Orlicz space
$L^H (G)$ equipped with the Luxemburg norm
\[
\|u\|_{L^\Phi(G)} = \inf \left\{\lambda >0: \int_G \Phi\left ( \frac{|u(x)|}{\lambda}\right)\, dx\le 1 \right\}
\]
is a Banach space.

We recall the following theorem from 
\cite[1.3. Theorem]{HH-S1}.

\begin{theorem}\label{thm:pointwise-maximal}
Let $\phi$ satisfy the conditions (1)--(5).
Let $\psi  :[0,\infty )\to [0,\infty )$ be defined as in \eqref{equ:psi}.
Let $1\le p<n$ be given.
Suppose that
there exists a continuous function $h:[0,\infty )\to [0, \infty)$ such that 
\begin{equation}\label{h_sum}
\sum_{k=1}^{\infty} \frac{(2^{-k}t )^n}{\psi (t 2^{-k})^{n-1}}
\le h(t ) \quad \text{ for all }\quad t>0\,.
\end{equation}
Let $\delta :(0,\infty )\to [0,\infty )$ be a continuous function and
let $H: [0,\infty )\to [0, \infty)$ be an N-function satisfying the $\Delta _2$-condition. Suppose that
there exists a finite constant $C_H$ such that the inequality 
\begin{equation}\label{sum}
H\left(h(\delta (t)) t + 
\psi (\delta (t) )^{1-n}(\delta(t)) ^{n(1-\frac{1}{p})} \right)\le
C_H t^p 
\end{equation}
holds for all $t>0$.
Let $G$ in $\Rn$ be an open set.
If $\| f \|_{L^p (G)} \le 1$, then 
there exists  a constant $C$ such that 
the inequality
\begin{equation}\label{main_riesz_inequality}
H\left ( \int_{G}  \frac{|f(y)|}{\psi(  |x-y|)^{n-1}} \, dy\right )
\le C (M f (x))^p
\end{equation}
holds for every $x\in \Rn$. Here the constant $C$ depends on $n$, $p$, $C_\varphi$, $C_H$, and the $\Delta_2$-constants of $\phi$ and $H$ only.
\end{theorem}

Our goal is to find a formula  which
would give all suitable functions $H$. Examples of some of these
functions were given in \cite[Section 6]{HH-S1}.

Here we do the preparations to find  $H$.
Assume that there exists $\alpha \in[1, {n}/{(n-1)})$ such that
$t^\alpha/\phi(t)$ is increasing for $t>0$. This yields that $t^\alpha/\psi(t)$ is increasing, too. Under this condition  inequality \eqref{h_sum} holds: Since
\[
\begin{split}
\frac{(2^{-k}t )^n}{\psi (t 2^{-k})^{n-1}} 
&= \frac{(2^{-k}t )^n}{(2^{-k}t )^{\alpha(n-1)}} \cdot \frac{(2^{-k}t )^{\alpha(n-1)}}{\psi (t 2^{-k})^{n-1}}\\ 
&\le (2^{-k}t )^{n- \alpha(n-1)} \frac{t^{\alpha(n-1)}}{\psi(t)^{n-1}} = 2^{-k(n-\alpha(n-1))} \frac{t^n}{\psi(t)^{n-1}},
\end{split}
\]
we have
\[
\sum_{k=1}^{\infty} \frac{(2^{-k}t )^n}{\psi (t 2^{-k})^{n-1}} \le C(n, \alpha) \frac{t^n}{\psi(t)^{n-1}},
\quad \text{where} \quad C(n, \alpha) = \frac{2^{\alpha(n-1)}}{2^n - 2^{\alpha(n-1)}}.
\]

Let us define  the functions $h$ and $\delta$ such that
\[
h(t) = C(n, \alpha) \frac{ t^n}{\psi(t)^{n-1}} \quad \text{and} \quad \delta(t) = t^{- \frac{p}{n}} \quad \text{for all} \quad t>0. 
\]
Then,
\[
\begin{split}
h(\delta (t)) t + 
\psi (\delta (t) )^{1-n}(\delta(t)) ^{n(1-\frac{1}{p})}
&= h \left( t^{- \frac{p}{n}}\right) t +  \psi \left(t^{- \frac{p}{n}} \right)^{1-n}\left(t^{- \frac{p}{n}} \right) ^{n(1-\frac{1}{p})}\\
&= \frac{C(n,\alpha) t^{-p}}{\psi\left(t^{- \frac{p}{n}} \right)^{n-1}} t  + \frac{t^{1-p}}{\psi \left(t^{- \frac{p}{n}} \right)^{n-1}}\\
&= \frac{(C(n, \alpha) +1) t^{1-p}}{\psi \left(t^{- \frac{p}{n}} \right)^{n-1}}.
\end{split}
\]
If we choose 
\[
F^{-1}(t) = \frac{ (C(n, \alpha) +1) (t^{1/p})^{1-p}}{\psi \left((t^{1/p})^{- \frac{p}{n}} \right)^{n-1}} = \frac{(C(n, \alpha) +1) t^{\frac1p-1}}{\psi \left(t^{- \frac{1}{n}} \right)^{n-1}}
\]
and assume that the inverse function of $F^{-1}$ exists, 
that is 
$(F^{-1})^{-1}=:F$ exists,
then 
\[
h(\delta (t)) t + 
\psi (\delta (t) )^{1-n}(\delta(t)) ^{n(1-\frac{1}{p})} = F^{-1}(t^p)
\]
and thus 
\[
F\left(h(\delta (t)) t + 
\psi (\delta (t) )^{1-n}(\delta(t)) ^{n(1-\frac{1}{p})}\right) = F \left(F^{-1}(t^p) \right) = t^p.
\]
Unfortunately,
there is a problem  with this function $F$  to be a suitable function $H$; namely, the function $F$ is not necessary convex. For example, if $n=2$, $\phi(t) = t^{\frac32}$, and $p=1.9$, then the function $F$ is not convex, see Figure~\ref{fig:not_convex}. The angle at the point $(1, F^{-1}(1))$ comes from the angle of  $\psi$ at the point $(1, \psi(1))$. 
Our main theorem, Theorem~\ref{thm:defn_H_intro} in Introduction, corrects this point: we show that there exists an $N$-function $H$ that is equivalent with $F$.

\begin{figure}[ht!]
\includegraphics[width=6cm]{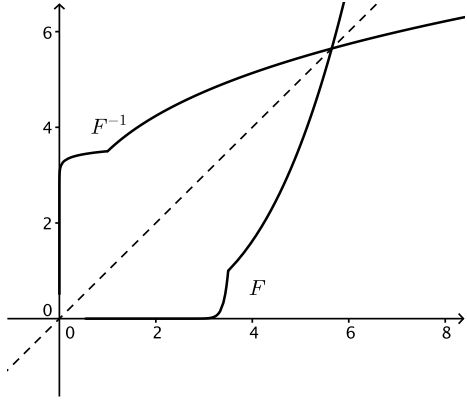}
\caption{The function $F$ is not necessary convex.}\label{fig:not_convex}
\end{figure}

\begin{proof}[Proof of Theorem~\ref{thm:defn_H_intro}]
Let us write that
\[
F^{-1}(t)  = \frac{t^{\frac1p-1}}{\psi \left(t^{- \frac{1}{n}} \right)^{n-1}}
\]
for $t>0$ and $F^{-1}(0)=0$.
Let us first show that $F^{-1}$ is strictly increasing.
Assume then that $0 < s <t$. The inequality $F^{-1}(s) < F^{-1}(t)$ is equivalent to the inequality
\[
\frac{\psi\left( \left(\frac1t \right)^{\frac1n} \right )^{n-1}}{\left( \frac1t\right)^{1- \frac1p}} <
\frac{\psi\left( \left(\frac1s \right)^{\frac1n} \right )^{n-1}}{\left( \frac1s\right)^{1- \frac1p}}.
\]
Recall that if $\phi$ satisfies the condition $(4)$, then $\psi$ does, too,  and the constant is the same for both functions. Thus by the condition (4) and the inequality $p< n$ we obtain
\[
\begin{split}
\frac{\psi\left( \left(\frac1t \right)^{\frac1n} \right )^{n-1}}{\left( \frac1t\right)^{1- \frac1p}} &= \left(\frac{\psi\left( \left(\frac1t \right)^{\frac1n} \right )}{\left( \frac1t\right)^{\frac1n}} \right)^{n-1} \left(\frac1t \right)^{\frac{n-1}n - 1 + \frac1p} \le \left(\frac{\psi\left( \left(\frac1s \right)^{\frac1n} \right )}{\left( \frac1s\right)^{\frac1n}} \right)^{n-1}\left(\frac1t \right)^{\frac{n-1}n - 1 + \frac1p} \\
&< \left(\frac{\psi\left( \left(\frac1s \right)^{\frac1n} \right )}{\left( \frac1s\right)^{\frac1n}} \right)^{n-1} \left(\frac1s \right)^{\frac{n-1}n - 1 + \frac1p}
= \frac{\psi\left( \left(\frac1s \right)^{\frac1n} \right )^{n-1}}{\left( \frac1s\right)^{1- \frac1p}}.
\end{split}
\]
Thus the function $F^{-1}$ is strictly increasing. This yields that the function $F$ exists and is  strictly increasing.

Let us show that $\lim_{t \to 0^+} F^{-1}(t) = 0$. Since $p<n$ we obtain
\[
\begin{split}
\lim_{t \to 0^+}F^{-1}(t)  = \lim_{t \to 0^+}\frac{t^{\frac1p-1}}{\psi \left(t^{- \frac{1}{n}} \right)^{n-1}} = \lim_{t \to 0^+} \phi(1)^{1-n} t^{\frac{n-1}{n}+ \frac1p -1} =0.
\end{split}
\]
Let us show that $\lim_{t \to \infty} F^{-1}(t) = \infty$. Since $t/\phi(t)$ is decreasing,  by the condition (4), and by  $p< n$ we obtain
\[
\lim_{t \to \infty}F^{-1}(t)  = \lim_{t \to \infty}\frac{t^{\frac1p-1}}{\psi \left(t^{- \frac{1}{n}} \right)^{n-1}}
=\lim_{t \to \infty} t^{\frac1p-\frac1n} \left (\frac{t^{- \frac1n}}{\psi \left(t^{- \frac{1}{n}} \right) }\right)^{n-1} \ge \lim_{t \to \infty}  \frac{t^{\frac1p-\frac1n}}{\phi (1 )^{n-1}}=  \infty.
\]
We have shown that $F^{-1}:[0, \infty) \to [0, \infty)$ is bijective.

Let us then study the condition  
\begin{equation}\label{equ:N/t_increasing}
\frac{F(s)}{s} < \frac{F(t)}{t} \quad\text{for}\quad 0<s<t.
\end{equation}  
Since $F^{-1}$ is a strictly increasing bijection,  inequality \eqref{equ:N/t_increasing}  is equivalent to
\[
\frac{s}{F^{-1}(s)} < \frac{t}{F^{-1}(t)}.
\]
 Since $t^\alpha/\phi(t)$ is  increasing, then $\phi(t)/t^\alpha$ is  decreasing and $\psi(t)/t^\alpha$ is  decreasing, too. We note that $1- \frac{\alpha(n-1)}{n}>0$, since $\alpha < \frac{n}{n-1}$. We obtain
\[
\begin{split}
\frac{s}{F^{-1}(s)} &=  s^{2-\frac1p} \psi\left(s^{-\frac1n}\right)^{n-1}
=  s^{2-\frac1p-\frac{\alpha(n-1)}{n}}  \left( \frac{\psi\left(s^{-\frac1n}\right)}{\left(s^{-\frac1n}\right)^\alpha} \right)^{n-1}\\
&=  s^{1-\frac1p + 1-\frac{\alpha(n-1)}{n}}  \left( \frac{\psi\left(s^{-\frac1n}\right)}{\left(s^{-\frac1n}\right)^\alpha} \right)^{n-1}
< t^{1-\frac1p + 1-\frac{\alpha(n-1)}{n}}  \left( \frac{\psi\left(t^{-\frac1n}\right)}{\left(t^{-\frac1n}\right)^\alpha} \right)^{n-1}
= \frac{t}{F^{-1}(t)}
\end{split}
\]
and thus inequality \eqref{equ:N/t_increasing} holds.

Let us then show that $F^{-1}(cs) \ge 2 F^{-1} (s)$ for all $s \ge 0$ with $c= 2^\frac{np}{n-p}$. 
The inequality $F^{-1}(cs) \ge 2 F^{-1} (s)$ is equivalent to 
\[
2\frac{\psi\left( \left(\frac1{cs} \right)^{\frac1n} \right )^{n-1}}{\left( \frac1{cs}\right)^{1- \frac1p}} \le
\frac{\psi\left( \left(\frac1s \right)^{\frac1n} \right )^{n-1}}{\left( \frac1s\right)^{1- \frac1p}}.
\]
By the condition $(4)$ of $\phi$  and the inequality $p< n$, we obtain
\[
\begin{split}
2\frac{\psi\left( \left(\frac1{cs} \right)^{\frac1n} \right )^{n-1}}{\left( \frac1{cs}\right)^{1- \frac1p}} &= 2 \left(\frac{\psi\left( \left(\frac1{cs} \right)^{\frac1n} \right )}{\left( \frac1{cs}\right)^{\frac1n}} \right)^{n-1} \left(\frac1{cs} \right)^{\frac{n-1}n - 1 + \frac1p}
=  \left(\frac{\psi\left( \left(\frac1{cs} \right)^{\frac1n} \right )}{\left( \frac1{cs}\right)^{\frac1n}} \right)^{n-1} \left(\frac1{s} \right)^{\frac{n-1}n - 1 + \frac1p}\\
&\le \left(\frac{\psi\left( \left(\frac1s \right)^{\frac1n} \right )}{\left( \frac1s\right)^{\frac1n}} \right)^{n-1} \left(\frac1s \right)^{\frac{n-1}n - 1 + \frac1p}
= \frac{\psi\left( \left(\frac1s \right)^{\frac1n} \right )^{n-1}}{\left( \frac1s\right)^{1- \frac1p}}.
\end{split}
\]
The inequality $F^{-1}(cs) \ge 2 F^{-1} (s)$ yields that $F$ satisfies the $\Delta_2$-condition. Let us write $F(t) =s$. Then $F^{-1}(s) =t$. Since $F$ is increasing, we have
\[
F(2t) = F(2F^{-1}(s)) \le F(F^{-1}(cs)) = cs = c F(t). 
\]

P.\ H\"ast\"o has shown in \cite[Proposition~5.1]{Hasto} that if  $f:[0, \infty) \to [0, \infty)$ satisfies the $\Delta_2$-condition and $x \mapsto {f(x)}/{x}$ is increasing, then $f$ is equivalent to a convex function.
Since $F$ satisfies inequality \eqref{equ:N/t_increasing} and the $\Delta_2$-condition, we obtain  that $F$ is equivalent to a convex function $H$.

Using 
$\lim_{t \to 0^+} F^{-1}(t) =0$ and the bijectivity,
we obtain
\[
\begin{split}
\lim_{t \to 0^+}\frac{F(t)}{t} &= \lim_{t \to 0^+}\frac{t}{F^{-1}(t)}
= \lim_{t \to 0^+} \frac{t\, \psi\left( \left(\frac1t \right)^{\frac1n} \right )^{n-1}}{\left( \frac1t\right)^{1- \frac1p}} = \lim_{t \to 0^+} \phi(1)^{n-1} t^{1-\frac1p + 1 - \frac{n-1}{n}} =0 
\end{split}
\]
and thus also $\lim_{t \to 0^+}\frac{H(t)}{t} =0$. This gives that $H$ is right continuous at the origin. Thus by convexity the function $H$ is continuous on $[0, \infty)$.

Since $\phi(t)/t^\alpha$ is decreasing and $\alpha < \frac{n}{n-1}$, we obtain 
\[
\begin{split}
\lim_{t \to \infty}\frac{F(t)}{t} &=
\lim_{t \to \infty}\frac{t}{F^{-1}(t)}
= \lim_{t \to \infty}  t^{2-\frac1p} \phi\left(t^{-\frac1n}\right)^{n-1}\\
& = \lim_{t \to \infty}  t^{2-\frac1p- \frac{\alpha(n-1)}{n}} \left( \frac{\phi\left(t^{-\frac1n}\right)}{\left(t^{-\frac1n}\right)^\alpha} \right)^{n-1} \ge  \lim_{t \to \infty}  t^{1-\frac1p+1- \frac{\alpha(n-1)}{n}} \left( \frac{\phi\left(1\right)}{1^\alpha} \right)^{n-1}= \infty.
\end{split}
\]
Since the functions $F$ and $H$ are equivalent, this yields that 
\[
\lim_{t \to \infty}\frac{H(t)}{t}= \infty. 
\] 
Thus we have shown that the function $H$ satisfies the conditions (N1) -- (N3).
\end{proof}

\begin{remark}\label{rem:H-pienilla-arvoilla}
Later it is crucial to us that  
\[
H^{-1}(t) \approx \frac{t^{\frac1p-1}}{\psi \left(t^{- \frac{1}{n}} \right)^{n-1}}
 = \frac{t^{\frac1p-1}}{\phi(1)^{n-1} \left(t^{- \frac{1}{n}} \right)^{n-1}}
 = \phi(1)^{1-n} t^{\frac{n-p}{np}}
\]
for $0<t\le 1$.
Namely, then for every $\phi$ the function $H$ satisfies
$H(t) \approx t^{\frac{np}{n-p}}$  whenever $0<t\le 1$.
\end{remark}

\begin{example}
Functions 
$\phi (t)=t^{\alpha}/\log^{\beta}(e+1/t)$,
$\alpha \in[1, \frac{n}{n-1})$ and $\beta \ge 0$, satisfy the assumptions of Theorem~\ref {thm:defn_H_intro}.
\end{example}

Now, the proof for our second main theorem, Theorem \ref{thm_main} in Introduction,
follows easily:

\begin{proof}[Proof of Theorem \ref{thm_main}]
Theorem \ref{thm:pointwise-maximal}  and Theorem \ref{thm:defn_H_intro}.
\end{proof}

As a corollary we obtain 
from Theorem \ref{thm_main} and Theorem~ \ref{thm:Riesz}:

\begin{corollary}\label{cor:funktiosta-gradientiin}
Let $1\le p<n$.
Let the function  $H$ be as in Theorem~\ref{thm:defn_H_intro}. 
If $D$ is a bounded $\phi$-cigar John domain with a constant $c_J$, then
there exit a constant $C$ and  a point $x_0 \in D$ such that
the point-wise estimate
\begin{equation*}%\label{equ:funktiosta-gradientiin}
H\left( \big|u(x) - u_{B(x_0,  \dist(x_0, \partial D))}\big|\right)  \le C (M |\nabla u| (x))^p
\end{equation*}
holds 
for all $u\in L^1_p(D)$ with $\|\nabla u\|_{L^p(D)} \le 1$ and for almost every $x\in D$. Here the constant $C$ depends on $n$, $p$, $C_H$, $C_H^{\Delta_2}$, $C_\phi^{\Delta_2}$,  $c_J$, $\phi(1)$ and  $\min\Big\{\diam(D), 1\Big\}$
only.
\end{corollary}

%%%%%%%%%%%%%%%%%%%%%%%%%%%%%%%%%%%%%%%%%%%%%%%%%%%%%%%%%%%%%%%%%%%%%%%%%%
\section{On embeddings}

Corollary
\ref{cor:funktiosta-gradientiin} is essential in the proofs of the following
Theorem \ref{thm:Sobolev-Poincare-p>1} and Theorem
\ref{thm:Sobolev-Poincare-p=1}.

\begin{theorem}[Bounded domain, $1<p<n$]\label{thm:Sobolev-Poincare-p>1}
Assume that $\phi$ satisfies the conditions $(1)$-- $(5)$, $C_\phi=1$ in the condition $(4)$, and there exists $\alpha \in[1, {n}/{(n-1)})$ such that
$t^\alpha/\phi(t)$ is increasing for $t>0$. 
Let $\psi$ be defined as in    \eqref{equ:psi}.
Let $D \subset \Rn$, $n \ge 2$,  be a bounded $\phi$-cigar John domain with a constant $c_J$. Let $1< p<n$.  
Then there exists an $N$-function $H$, that satisfies $\Delta_2$-condition and  
\[
H^{-1}(t)  \approx \frac{t^{\frac1p-1}}{\psi \left(t^{- \frac{1}{n}} \right)^{n-1}} \, \mbox{ for all } t>0\,,
\]
and there exists a constant $C<\infty$ such that the inequality
\[
\| u-  u_D \|_{L^H(D)} \le C \|\nabla u \|_{L^p (D)},
\]
holds
for every $u \in L^1_p (D)$. Here the constant $C$ depends  on $n$, $p$, $C_H^{\Delta_2}$, $C_\phi^{\Delta_2}$, $c_J$ and $\min\{\diam(D), 1\}$ only.
\end{theorem}

\begin{proof}
%Theorem~\ref{thm:John-John} yields that $D$ is a bounded $\psi$-John domain.
Assume that $\|\nabla u\|_{L^p(D)} \le 1$.
Corollary \ref{cor:funktiosta-gradientiin}  yields that
\[
H\left( \big|u(x) - u_{B(x_0,  \dist(x_0, \partial D))}\big|\right)  \le C (M |\nabla u| (x))^p,
\]
where the constant $C$ depends  on $n$, $p$, $C_H^{\Delta_2}$, $C_\phi^{\Delta_2}$, $c_J$, and $\min\{\diam(D), 1\}$ only.
By integrating over $D$ and using the fact that the maximal operator is bounded whenever  $1<p<n$, we obtain that
\begin{equation*}%\label{equ:H-bounded-1}
\begin{split}
\int_D H\left( \big|u(x) - u_{B(x_0,  \dist(x_0, \partial D))}\big|\right)  \, dx &\le C \int_D (M |\nabla u| (x))^p \, dx \\
&\le C\int_D |\nabla u (x)|^p \, dx \le C.
\end{split}
\end{equation*}
This yields that the inequality
\[
 \| u - u_{B(x_0,  \dist(x_0, \partial D))}\|_{L^H(D)} \le C
 \]
 holds
  for every $u\in L^1_p(D)$ with $\|\nabla u\|_{L^p(D)} \le 1$. 
By applying this inequality to the function $u/ \|\nabla u\|_{L^p(D)}$ we obtain that
\[
\| u - u_{B(x_0,  \dist(x_0, \partial D))}\|_{L^H(D)} \le C \|\nabla u\|_{L^p(D)}.
\]
We may assume w.l.o.g. that  $\|\nabla u\|_{L^p(D)} \neq 0$.
Let denote $B = B(x_0,  \dist(x_0, \partial D))$. By the triangle inequality
\[
\|u- u_D\|_{L^H(D)} \le \|u- u_{B}\|_{L^H(D)} + \|u_{B}- u_D\|_{L^H(D)}.
\]
Here,
\[
\begin{split}
\|u_{B}- u_D\|_{L^H(D)} &= |u_{B}- u_D|
\, \|1 \|_{L^H(D)}\le
\frac{\|1 \|_{L^H(D)}}{|D|} \|u-u_{B}\|_{L^1(D)}\\
&\le C \frac{\|1 \|_{L^H(D)} \|1 \|_{L^{H^*}(D)}}{|D|} \|u-u_{B}\|_{L^H(D)}
\end{split}
\]
where $H^*$ is the conjugate function of $H$ and $C$ is the constant in Hölder's inequality. 

Next we show that $\|1 \|_{L^H(D)} \|1 \|_{L^{H^*}(D)} \approx |D|$. Since the function $H$ is continuous and strictly increasing, there exists a unique $\lambda>0$ such that
\[
H\left(1/{\lambda} \right) |D| = \int_D H\left(1/{\lambda} \right) \, dx =1
\]
i.e. $\lambda = \|1 \|_{L^H(D)}$. By solving $\lambda$ we obtain
\[
\|1 \|_{L^H(D)} = \frac1{H^{-1}\left( \frac1{|D|}\right)}.
\]
Similarly, we obtain
\[
\|1 \|_{L^{H^*}(D)} = \frac1{(H^*)^{-1}\left( \frac1{|D|}\right)}.
\] 
Since
\[
t \le H^{-1}(t) (H^{*})^{-1}(t) \le 2t
\]
for all $t \ge 0$, see for example \cite[Lemma 2.6, p.~56]{DieHHR11}, we obtain that
\[
\|1 \|_{L^H(D)}\|1 \|_{L^{H^*}(D)} = \frac1{H^{-1}\left( \frac1{|D|}\right) (H^*)^{-1}\left( \frac1{|D|}\right)} \le |D|.
\]

Hence, we have shown that
\begin{equation*}%\label{equ:S-P-inequ}
\| u - u_D\|_{L^H(D)} \le C  \|\nabla u\|_{L^p(D)}
\end{equation*}
for every $u \in L^1_p (D)$.
\end{proof}

\begin{example}\label{exm:H-funktio}
Let us choose that $\phi(t) = t^s$, $s \in[1, \frac{n}{n-1})$. 
We have calculated  in Remark~\ref{rem:H-pienilla-arvoilla} that  
for every $\phi$ the function $H$ satisfies
$H(t) \approx t^{\frac{np}{n-p}}$  whenever $0<t\le 1$.
If $t>1$, then
\[
H^{-1}(t) \approx \frac{t^{\frac1p-1}}{\psi \left(t^{- \frac{1}{n}} \right)^{n-1}}
 = \frac{t^{\frac1p-1}}{\phi \left(t^{- \frac{1}{n}} \right)^{n-1}} = 
 \frac{t^{\frac1p-1}}{ \left(t^{- \frac{1}{n}} \right)^{s(n-1)}}
 = t^{\frac{n-np+sp(n-1)}{np}}
\]
and thus we have that $H(t) \approx t^{\frac{np}{n-np+sp(n-1)}}$  for $t> 1$.
\end{example}

\begin{theorem}[Bounded domain, $p=1$]\label{thm:Sobolev-Poincare-p=1}
Assume that the function $\phi$ satisfies the conditions $(1)$-- $(5)$, $C_\phi=1$ in the condition $(4)$, and there exists $\alpha \in[1, {n}/{(n-1)})$ such that
$t^\alpha/\phi(t)$ is increasing for $t>0$. 
Let $\psi$ be defined as in (\ref{equ:psi})
Let $D \subset \Rn$, $n \ge 2$, be a bounded $\phi$-cigar John domain with a constant $c_J$.   Then there exists an $N$-function $H$, that satisfies $\Delta_2$-condition and  
\[
H^{-1}(t)  \approx \frac{1}{\psi \left(t^{- \frac{1}{n}} \right)^{n-1}}\, \mbox{ for all } t>0\,,
\]
such that the inequality
\[
\| u- u_D \|_{L^H(D)} \le C \|\nabla u \|_{L^1 (D)},
\]
holds
for some constant $C$ and
for every $u \in L^1_p(D)$. Here the constant $C$ depends  on $n$, $C_H^{\Delta_2}$, $C_\phi^{\Delta_2}$ and $c_J$ only.
\end{theorem}

\begin{proof}%[Proof of Theorem~\ref{thm:Sobolev-Poincare-p=1}]
%Theorem~\ref{thm:John-John} yields that $D$ is a bounded $\psi$-John domain. Let $x_0$ be the John center.
Let us consider functions $u\in L^1_1 (D)$ such that
$\| \nabla u \|_{L^1(D)} \le 1$.  The center ball 
$B(x_0, \dist(x_0, \partial D ))$ is written as $B$.  In the proof of Theorem~\ref{thm:John-John} we had chosen $x_0$ so that $\dist(x_0, \partial D ) \ge \psi( \frac14 \diam(D))/c_J$.
We show  that there exists a constant $C<\infty$ such that the inequality
\begin{equation}\label{enough}
\int_D H(|u(x)-u_B|) \, dx \le C
\end{equation}
holds whenever $\|\nabla u\|_{L^1(D)} \le 1$. This yields  the claim as in the proof of Theorem~\ref{thm:Sobolev-Poincare-p>1}.

Since $H$ is increasing, we first estimate
\[
\int_D H(|u(x)-u_B|) \, dx \le \sum_{j \in \Z} \int_{ \{x\in D: 2^j < |u(x) - u_B| \le 2^{j+1}\}} H(2^{j+1}) \, dx.
\]
Let us define 
\[
v_j(x) = \max\bigg\{0, \min\Big\{|u(x) - u_B| - 2^j, 2^j\Big\}\bigg\}
\]
for all $x\in D$. 
\noindent
If $x \in \{x\in D :2^j < |u(x) - u_B|\le 2^{j+1}\}$, then $v_{j-1}(x) \ge 2^{j-1}$. 
We obtain
\begin{equation}\label{equ:main-2}
\int_D H(|u(x)- u_B|) \, dx \le \sum_{j \in \Z} \int_{\{x\in D :v_j(x) \ge 2^{j}\}} H(2^{j+2}) \, dx.
\end{equation}
By the triangle inequality %and Theorem~\ref{thm:pointwise-John} 
we have
\[
\begin{split}
v_j(x) &= |v_j(x) - (v_j)_B + (v_j)_B| \le |v_j(x) - (v_j)_B| + |(v_j)_B|.
%\\ 
%&\le C \int_{D}  \frac{|\nabla v_j(y)|}{\varphi \big( |x-y|\big)^{n-1}} \, dy +  |(v_j)_B|
\end{split}
\]
%for almost every $x\in D$.
By the $(1,1)$-Poincar\'e inequality in a ball $B$, \cite[Section 7.8]{Gilbarg-Trudinger}, there exists a constant $C(n)$ such that
\begin{equation*}%\label{equ:p=1_vakio_ongelma}
\begin{split}
|(v_j)_B|&= (v_j)_B = \vint_B v_j(x) \, dx \le \vint_B |u(x) - u_B| \, dx\\  
&\le C(n) |B|^{\frac1n} \vint_B |\nabla u(x)| \, dx \le C(n) |B|^{\frac1n-1}.
\end{split}
\end{equation*}
%Thus, by the definition of $B$ the number $|(v_j)_B|$ is bounded by a constant depending  on $n$ and the distance between the John center and the boundary
%of $D$ only. Moreover the constant tends to zero as the radius of $B$ tends to infinity.

%We write
%\[
%I_{\varphi}(\nabla v_j)(x) = \int_{D}  \frac{|\nabla v_j(y)|}{\varphi \big( |x-y|\big)^{n-1}} \, dy.
%\]
We continue to estimate the right hand side of inequality \eqref{equ:main-2}
\begin{equation}\label{equ:main-3}
\begin{split}
&\int_D H(|u(x)- u_B|) \, dx \le \sum_{j \in \Z} \int_{\{x\in D: |v_j(x) - (v_j)_B|+ C|B|^{-1} \ge 2^{j}\}} H(2^{j+2}) \, dx\\
&\quad \le \sum_{j \in \Z} \int_{\{x\in D : |v_j(x) - (v_j)_B| \ge 2^{j-1}\}} H(2^{j+2}) \, dx +\sum_{2^{j-1} \le C(n) |B|^{\frac1n-1} } \int_{D} H(2^{j+2}) \, dx\\
&\quad \le \sum_{j \in \Z} \int_{\{x\in D : |v_j(x) - (v_j)_B| \ge 2^{j-1}\}} H(2^{j+2}) \, dx +\sum_{j=-\infty}^{j_0} \int_{D} H(2^{j+2}) \, dx,
\end{split}
\end{equation}
where $j_0 =  \lceil\log(C(n) |B|^{\frac1n-1})\rceil$.

Assume first that $\diam (D)$ is so large that $j_0 \le -2$.
When $t<1$, then $\psi(t^{-1/n}) = \phi(1) t^{-1/n}$  by \eqref{equ:psi} and thus
\[
H^{-1}(t) = \frac{1}{\psi(t^{-1/n})^{n-1}} = \phi(1)^{1-n} t^{\frac{n-1}{n}}.
\] 
Thus for $t<1$ we obtain that  $H(t) \approx t^{\frac{n}{n-1}}$.
 This yields that 
\begin{equation}\label{equ:main-4a}
\begin{split}
\sum_{j=-\infty}^{j_0} \int_{D} H(2^{j+2}) \, dx
&\approx |D| \sum_{j=-\infty}^{\lceil\log(C |B|^{\frac1n-1})\rceil} 2^{\frac{n(j+2)}{n-1}} \le C|D| 2^{\frac{n}{n-1} \cdot \lceil\log(C |B|^{\frac1n-1})\rceil}\\
& \le C|D| |B|^{\frac{n}{n-1}( \frac1n -1)} = C |D| |B|^{-1}\\
&\le C \frac{\diam(D)^n}{(\psi( \frac14 \diam(D))/c_J)^n}.
\end{split}
\end{equation}
This constant does not blow up when $\diam(D) \to \infty$:
\[
\frac{\diam(D)^n}{(\psi( \frac14 \diam(D))/c_J)^n} \to \frac{4^n c_J^n}{\phi(1)^n} \quad \text{as} \quad \diam(D) \to \infty.
\]
Assume then that $\diam(D)$ is small. 
This yields that for every $j_0 \in  \Z$ the sum $\sum_{j=-2}^{j_0}H(2^{j+2})$ is finite and depends on $j_0$.
We obtain 
\begin{equation}\label{equ:main-4b}
\sum_{j=-\infty}^{j_0} \int_{D} H(2^{j+2}) \, dx
\le \sum_{j=-\infty}^{-2} \int_{D} H(2^{j+2}) + \sum_{j=-2}^{j_0}H(2^{j+2})< \infty.
\end{equation}

Then, we will find an upper bound for the sum
\begin{equation*}
\sum_{j \in \Z} \int_{\{x\in D : |v_j(x) - (v_j)_B| \ge 2^{j-1}\}} H(2^{j+2}) \, dx\,. 
\end{equation*}
Since $\|\nabla v_j\|_{L^1(D)} \le \|\nabla u\|_{L^1(D)} \le 1$, Corollary \ref{cor:funktiosta-gradientiin} yields that
\[
\begin{split}
\sum_{j \in \Z} \int_{\{x\in D: |v_j(x) - (v_j)_B| \ge 2^{j-1}\}} H(2^{j+2}) \, dx 
&= \sum_{j \in \Z} \int_{\{x\in D: H(|v_j(x) - (v_j)_B|) \ge H(2^{j-1})\}} H(2^{j+2}) \, dx \\ 
&\le  \sum_{j \in \Z} \int_{\{x\in D: C  M|\nabla v_j|(x) \ge H(2^{j-1})\}} H(2^{j+2}) \, dx.
\end{split}
\]
We choose for every $x \in \{x\in D: C  M|\nabla v_j|(x) \ge H(2^{j-2})\}$ a ball $B(x, r_x)$,
centered at $x$ and with radius $r_x$ depending on $x$, such that
\[
C \vint_{B(x,r_x)} |\nabla v_j (y)| \, dy \ge \frac12 H(2^{j-1})
\]
with the understanding that $|\nabla v_j|$ is zero outside $D$.
By the Besicovitch covering theorem
(or the $5$-covering theorem) we obtain a subcovering $\{B_k\}_{k=1}^{\infty}$ so that we may estimate by the $\Delta_2$-condition of $H$
\[
\begin{split}
&\sum_{j \in \Z} \int_{\{x\in D: |v_j(x) - (v_j)_B| \ge 2^{j-1}\}} H(2^{j+2}) \, dx
\le \sum_{j \in \Z} \sum_{k=1}^\infty  \int_{B_k}H(2^{j+2}) \, dx\\
 & \quad\le \sum_{j \in \Z} \sum_{k=1}^\infty  |B_k| H(2^{j+2}) \le \sum_{j \in \Z} \sum_{k=1}^\infty C |B_k|  \frac{H(2^{j+2})}{H(2^{j-1})} \vint_{B_k} |\nabla v_j (y)| \, dy\\
&\quad \le C \sum_{j \in \Z} \int_{D} |\nabla v_j (y)| \, dy.
%&\quad\le C \int_{D} |\nabla u (y)| \, dy \le C.
\end{split}
\]
Let $E_j= \{x\in D :2^j < |u(x) - u_B|\le 2^{j+1}\}$.
Since $|\nabla v_j|$ is zero  almost everywhere in  $D\setminus E_j$ and
$|\nabla u (x)| = \sum_j |\nabla v_j (x)| \chi_{E_j}(x)$ for almost every $x \in D$, we obtain
\begin{equation}\label{equ:main-5}
\sum_{j \in \Z} \int_{\{x\in D: |v_j(x) - (v_j)_B| \ge 2^{j-1}\}} H(2^{j+2}) \, dx
\le C \int_{D} |\nabla u (y)| \, dy \le C.
\end{equation}
Estimates \eqref{equ:main-3}, \eqref{equ:main-4a}, \eqref{equ:main-4b}
and \eqref{equ:main-5} imply inequality
\eqref{enough}.
\end{proof}

\begin{remark}
Corollary 
\ref{corollary_bdd}
in Introduction follows from  Theorem \ref{thm:Sobolev-Poincare-p>1}
and Theorem  \ref{thm:Sobolev-Poincare-p=1}.
\end{remark}

\begin{remark}
In Theorem~\ref{thm:Sobolev-Poincare-p=1}  the $N$-function  $H$ is the best possible in a sense that it cannot be replaced by  any  $N$-function $K$ that satisfies the $\Delta_2$-condition and
\[
\lim_{t \to \infty} \frac{K(t)}{H(t)} = \infty.
\] 
In \cite[Theorem 7.2]{HH-S1} we have shown that the corresponding embedding  in Theorem~\ref{thm:Sobolev-Poincare-p=1}
does not hold if 
\[
\lim_{t \to 0^+} t^n K \left( \frac{1}{\phi(t)^{n-1}}  \right)=\infty.
\]
This is valid for this function $K$. 
By the definitions of $H^{-1}$ and $\psi$ we obtain that
\[
%\begin{split}
\lim_{t \to 0^+} t^n K \left( \frac{1}{\phi(t)^{n-1}}  \right)
= \lim_{s \to \infty} \frac1s K \left( \frac{1}{\phi\left(s^{-\frac1n}\right)^{n-1}}  \right) 
%&= \lim_{s \to \infty} \frac1s L \left( H^{-1}(s)  \right)
= \lim_{s \to \infty} \frac{K \left( H^{-1}(s)  \right)}{H \left( H^{-1}(s)  \right)}= \infty,
%\end{split}
\]
and thus there does not exists a constant $c$ such that
\[
\| u- u_D \|_{L^K(D)} \le c \|\nabla u \|_{L^1 (D)},
\]
for every $u \in L^1_p(D)$.
\end{remark}

\begin{theorem}[Unbounded domain, $1 \le p<n$]\label{thm:Sobolev-Poincare-unbounded}
Assume that the function $\phi$ satisfies the conditions $(1)$-- $(5)$, $C_\phi=1$ in the condition $(4)$, and there exists $\alpha \in[1, {n}/{(n-1)})$ such that
$t^\alpha/\phi(t)$ is increasing for $t>0$. 
Let the function $\psi$ be defined as in \eqref{equ:psi}.
Let $D$ in $\Rn$, $n \ge 2$, be an unbounded domain that satisfies the following conditions:
\begin{itemize}
\item[(a)] $D  = \cup_{i=1}^\infty D_i$, where $|D_1|>0$;
\item[(b)] $\overline D_i \subset D_{i+1}$ for each $i$;
\item[(c)] each $D_i$ is a bounded $\phi$-cigar John domain with a constant $c_J$.
\end{itemize}
Let $1\le p<n$.  Then, there exists an $N$-function $H$, that satisfies $\Delta_2$-condition and  
\[
H^{-1}(t)  \approx \frac{t^{\frac1p-1}}{\psi \left(t^{- \frac{1}{n}} \right)^{n-1}} \mbox{ for all } t>0\,,
\]
and there exits a constant $C$ such that the inequality
\[
\inf_{b \in \R} \| u- b \|_{L^H(D)} \le C \|\nabla u \|_{L^p (D)},
\]
holds
for every $u  \in L^1_p(D)$. Here the constant $C$ depends  on $n$, $p$, $C_H^{\Delta_2}$, $C_\phi^{\Delta_2}$ and $c_J$ only.
\end{theorem}

The proof follows the proof of \cite[Theorem~4.1]{H-S92}.

\begin{proof}
By Theorems \ref{thm:Sobolev-Poincare-p>1} and \ref{thm:Sobolev-Poincare-p=1}
there exits a constant $C$ such that the inequality
\begin{equation}\label{equ:S-P-inequ}
\| u- u_{D_i} \|_{L^H(D_i)} \le C \|\nabla u \|_{L^p (D_i)}
\end{equation}
holds for each $D_i$ and all $u \in L^1_p(D)$.
The constant $C$ does not blow up when the diameter 
of $D_i$ tends to infinity. In the case $1<p<n$ this is clear. In the case $p=1$, we refer to the discussion after 
\eqref{equ:main-4a}
in the proof of Theorem~\ref{thm:Sobolev-Poincare-p=1}. The constant depends on $D_1$ but this does not cause a problem.

When $\| \nabla u\|_{L^p(D)} \le 1$  inequality \eqref{equ:S-P-inequ} yields that 
there exists a constant $C<\infty$
such that the inequality
\[
\int_{D_i} H(|u(x) - u_{D_i}|) \, dx  \le C,
\]
holds; here the constant $C$ is independent of $i$. 

Let us write
\[
u_i =  \vint_{D_i} u(x) \, dx =  \frac1{|D_i|}\int_{D_i} u(x) \, dx.
\]
The triangle inequality yields that
\[
|u_i| \le   \vint_{D_1} |u(x) - u_i| \, dx+  \vint_{D_1} |u(x)| \, dx.
\]
Since $D_i$ satisfies inequality \eqref{equ:S-P-inequ}, we have $u \in L^H(D_1)\subset L^1(D_1)$ and thus the second term is finite.
Again, by inequality \eqref{equ:S-P-inequ} we obtain that 
\[
\begin{split}
\vint_{D_1} |u(x) - u_i| \, dx & \le
\frac{C \|1\|_{L^{H^*}(D_1)}}{|D_1|}   \| u - u_{D_i}\|_{L^H(D_1)}\\
&\le \frac{C \|1\|_{L^{H^*}(D_1)}}{|D_1|}   \| u - u_{D_i}\|_{L^H(D_i)} \le 
  \frac{C \|1\|_{L^{H^*}(D_1)}}{|D_1|} \| \nabla u\|_{L^p(D_i)}\\
&\le  \frac{C \|1\|_{L^{H^*}(D_1)}}{|D_1|} \| \nabla u\|_{L^p(D)}< \infty.
\end{split}
\]
Thus the real number sequence $(u_i)$ is bounded and hence there exists a convergent subsequence  $(u_{i_j})$ and $b \in \R$ such that $u_{i_j} \to b$. 

Since $H$ is continuous,
\[
\lim_{j \to \infty} \chi_{D_{i_j}} H(|u(x) - u_{i_j}|)
= \chi_{D} H(|u(x) - b|).
\]
Fatou's lemma and the modular form of \eqref{equ:S-P-inequ} yield that
\[
\begin{split}
\int_D H(|u(x) -b|) \, dx & \le \liminf_{j \to \infty} \int_D \chi_{D_{i_j}} H(|u(x) - u_{i_j}|)\, dx\\
&= \liminf_{j \to \infty} \int_{D_{i_j}} H(|u(x) - u_{i_j}|)
\le \liminf_{j \to \infty} C= C
\end{split}
\]
for every  $u \in L^1_{\textup{loc}}(D)$ with $\|\nabla u\|_{L^p(D)}\le 1$. This yields that 
there exists a constant $C$ such that the inequality
\[ 
\| u - b\|_{L^H(D)} \le C
\]
holds
for every $u\in L^1_p(D)$ with $\|\nabla u\|_{L^p(D)} \le 1$. The claim follows by applying this inequality to the function $u/ \|\nabla u\|_{L^p(D)}$.
\end{proof}

\begin{example}
Let the function $\phi$ be defined as in Theorem~\ref{thm:Sobolev-Poincare-unbounded}.
The following unbounded domains satisfy the assumptions of Theorem~\ref{thm:Sobolev-Poincare-unbounded}.

(a) $\left\{(x', x_n)\in\Rn: x_n \ge 0 \quad \text{and} \quad |x'|< \psi(x_n) \right\}$.

(b) $\Rk \setminus \left( \{(x,\phi(x)) \in \Rk: 0\le x \le 1\} \cup \{(x,-\phi(x)) \in \Rk: 0\le x \le 1\} \right)$.

\end{example}

%%%%%%%%%%%%%%%%%%%%%%%%%%%%%%%%%%%%%%%%%%%%%%%%%%%%%%%%%%%%%%%%%%%%%%%%%%%%%%%
\section{On Poincar\'e inequalities}

As a special case we recover results for Poincar\'e domains.
We recall that a bounded domain $D$ is called a $(q,p)$-Poincar\'e domain, where $q,p \in[1, \infty)$, if
there is a constant $C<\infty$ such that the inequality
\begin{equation}\label{Poincare}
\| u - u_D \|_{L^q(D)} \le C\| \nabla u\|_{L^p(D)}
\end{equation}
holds for all $u \in W^{1,p}(D)$. 
Inequality \eqref{Poincare} is the $(q,p)$-Poincar\'e inequality.
We note that for a bounded domain  $D$ inequality 
\eqref{Poincare}
holds if and only the inequality 
\[
\inf_{b\in \R} \| u - b \|_{L^q(D)} \le C_1 \| \nabla u\|_{L^p(D)}
\]
holds, the constants
$C$ and $C_1$ depend on each other and $|D|$ only.
Let us recall  results for bounded $\phi$-John domains in the case $\phi(t) = t^s$,  for a fixed $s\ge 1$. A bounded $t^s$-John domain is usually  called $s$-John domain.  A bounded $s$-John domain is a $(p,p)$-Poincar\'e domain  whenever $s \in \left[1, \frac{n+ p-1}{n-1}\right)$,
\cite[Theorem 10]{SmiS90}.
So, a bounded $s$-John domain is a $(p,p)$-Poincar\'e domain for all $p \ge 1$ if $s \in \left[1, \frac{n}{n-1}\right)$.
A bounded $s$-John domain is a $(1,p)$-Poincar\'e domain  if $s\in \left(1, \frac{p(n-\lambda +1)+ \lambda -1}{n-1} \right)$, where $\lambda\in[n-1,n]$ is the Minkowski dimension of the boundary of the domain, \cite[Theorem 1.3]{ HH-SV}.
A bounded $s$-John domain is a $\left(\frac{np}{s(n-1) - p +1},p \right)$-Poincar\'e domain for every $1 \le p < s(n-1) +1$ if $s \in \left[1, \frac{n}{n-1} \right]$;\cite[Corollaries 5 and 6]{HajK98}\,,
\cite[Theorem 2.3]{KilM00}.
The exponent $\frac{np}{s(n-1) - p +1}$ is the best possible in the class of bounded  $s$-John domains, we refer to  \cite[p. 442]{HajK98}.
Our Theorems~\ref{thm:Sobolev-Poincare-p>1} and \ref{thm:Sobolev-Poincare-p=1} 
with $\phi (t)=t^s$ give that a bounded $s$-John domain is a $\left(\frac{np}{n-np+sp(n-1)}, p \right)$-Poincar\'e domain if $1 \le p <n$ and $s \in \left[1, \frac{n}{n-1} \right)$.  Thus our result is optimal in the case $p=1$. On the other hand, our method does not cover the case $s=\frac{n}{n-1}$.
Note that our proof totally differs from the previous proofs 
in \cite[Theorem 10]{SmiS90}, 
\cite[Corollaries 5 and 6]{HajK98}, and   \cite[Theorem 2.3]{KilM00}.

%%%%%%%%%%%%%%%%%%%%%%%%%%%%%%%%%%%%%%%%%%%%%%%%%%%%%%

%%%%%%%%%%%%%%%%%%%%%%%%%%%%%%%%%%%%%%%%%%%%%%%%%%%%%%%%%%%%%%%%%%%%

%%%%%%%%%%%%%%%%%%%%%%%%%%%%%%%%%%%%%%%%%%%%%%%%%%%%%%%%%%

\section{Lebesgue space cannot be  a target space}

In this section we give an example which shows that for certain unbounded $\phi$-cigar John domains the target space cannot be a Lebesgue space. 
The idea is that at near the infinity the target space should be $L^{np/(n-p)}$ but local structure of the domain may not allow so good integrability. 
We assume a priori that the function $\phi$ has the properties (1)--(5). 
Later on we give extra conditions to the function $\phi$. 
%In Remark~\ref{rem:H-pienilla-arvoilla} we noticed that $H(t) \approx t^{\frac{np}{n-p}}$  for $0<t\le 1$ and if $\phi(t) = t^s$, then
%$H(t) \approx t^{\frac{np}{n-np+sp(n-1)}}$  for $t> 1$

We construct a mushrooms-type domain. Let $(r_m)$ be a decreasing sequence of positive real numbers converging to zero. 
Let $Q_{m}$, $m=1,2,\dots$, be a closed cube in $\Rn$ with side length $2r_{m}$. Let $P_{m}$,
$m=1,2,\dots$,
be a closed rectangle in $\Rn$ which has side length $r_{m}$ for one side and $2\varphi (r_{m})$ for the remaining $n-1$ sides.  Let $Q$ be the first quarter of the space i.e. all coordinates of the points in $Q$ are positive. 
We attach $Q_{m}$ and $P_{m}$ together creating 'mushrooms' which we then attach, as pairwise disjoint sets, to the side $\{(0, x_2, \ldots, x_n): x_2, \ldots, x_n >0\}$ of $Q$ so that the distance from the mushroom to the origin is at least $1$ and at most $4$, see Figure~\ref{fig:domain}.  
We assume that a priori the function $\phi$ has the properties (1)--(5), but 
we have to assume here also that $\varphi(r_m) \le r_m$.
We need copies of the mushrooms.  By an isometric mapping we transform these mushrooms onto the side $\{(x_1, 0, \ldots, x_n): x_1, x_3, \ldots, x_n >0\}$ of $Q$  and
denote them by $Q_m^*$ and $P_m^*$. So again the distance from the mushroom to the origin is at least $1$ and at most $4$.
We define
\begin{eqnarray}\label{eq:mushroon-domain}
G=\textrm{int}\left(Q \cup\bigcup_{m=1}^{\infty}\Big(Q_{m}\cup P_{m}\cup Q^{*}_{m}\cup P^{*}_{m}\Big)\right).
\end{eqnarray}
See Figure~\ref{fig:domain}. We omit  a short calculation which shows that $G$ is a $\phi$-cigar John domain. 

\begin{figure}[ht!]
\includegraphics[width=11 cm]{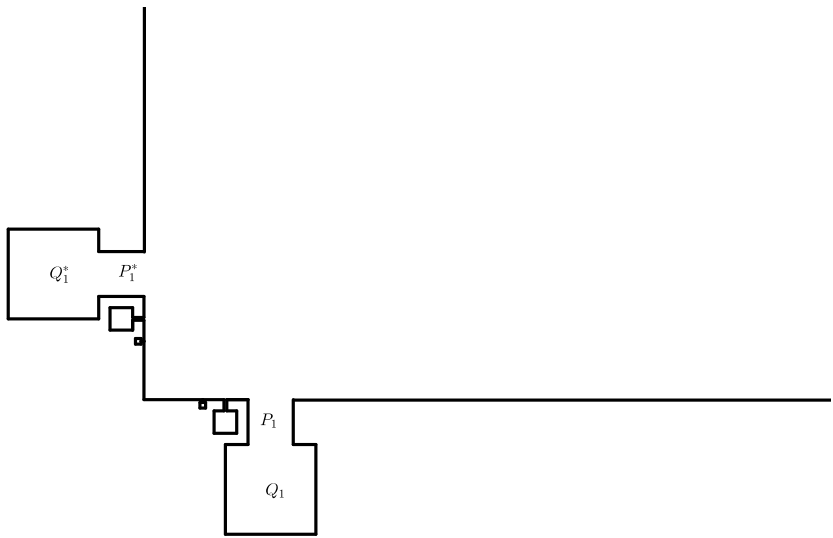}
\caption{Unbounded $\phi$-cigar John domain.}\label{fig:domain}
\end{figure}

Let us define a sequence of piecewise linear continuous functions $(u_k)_{k=1}^{\infty}$ by setting
\begin{equation*}
u_{k}(x):=
\begin{cases}
F(r_k)& \textrm{in } Q_{k}, \\
-F(r_k) & \textrm{in } Q^{*}_{k},\\
0 & \textrm{in} Q_0,
\end{cases}
\end{equation*}
where the function $F$ will be  given in \eqref{F}.
Then the integral average of $u_{k}$ over $G$ is zero
for each $k$. 

The gradient of $u_k$ differs from zero in  $P_m \cup P_m^*$ only and 
\[
|\nabla u_k(x)| = \frac{F(r_m)}{r_m}, \textrm{ when }  x\in P_m \cup P_m^* \,.
\] 
Note that
\begin{equation*}
\int_G |\nabla u_k(x)|^p \, dx = 2\int_{P_m}
\biggl(\frac{F(r_m)}{r_m}\biggr)^p=
2r_m\left(\varphi(r_m)\right)^{n-1}\frac{F(r_m)^p}{r_m^p}\,.
\end{equation*}
We require that
\begin{equation*}
\int_G |\nabla u_k(x)|^p \, dx =1\,.
\end{equation*}
Hence, we define
\begin{equation}\label{F}
F(r_m)=\biggl(\frac{r_m^{p-1}}{2\varphi (r_m)^{n-1}}\biggr)^{1/p}\,.
\end{equation}
Let $H$ be an  $N$-function. Then,
\begin{equation*}
\begin{split}
\inf_{b\in\R} \int_G H(|u_k(x) - b | )\, dx &\ge  \inf_{b\in\R} \biggl( |Q_m| \cdot |H(F(r_m) -b)| +  |Q_m^*|\cdot |H(-F(r_m) -b)| \biggr)\\
&\ge r_m^nH(F(r_m))\,.
\end{split}
\end{equation*}
Hence,   we have 
\begin{equation*}%\label{H}
\begin{split}
r^n_mH(F(r_m))=
r_m^nH\biggl(\biggl(\frac{r_m^{p-1}}{ 2 \varphi (r_m^{n-1})}\biggr)^{1/p}\biggr)
&\ge r_m^nH\biggl(\frac12 \biggl(\frac{r_m^{p-1}}{  \varphi (r_m^{n-1})}\biggr)^{1/p}\biggr).
\end{split}
\end{equation*}
Thus, there does not exist a
positive constant $C$ such that the inequality $\inf_b \|u-b \|_{L^H(G)} \le C \| \nabla u\|_{L^p(G)}$
could hold for all $u$ from the appropriate space if
\[
\lim_{t \to 0^+} t^nH\biggl(\frac12 \biggl(\frac{t^{p-1}}{ \varphi (t)^{n-1}}\biggr)^{1/p}\biggr) = \infty.
\]
Assume that $\lim_{t \to 0^+} t/\phi(t) = \infty$. If $H(t)= t^q$, then we obtain that the inequality does not hold if 
\begin{equation}\label{equ:ehto1}
q \ge  \frac{np}{n-p}.
\end{equation}

Assume then that we have a sequence $(s_j)$  of positive numbers going to infinity. For each $s_j$ we may choose points $x(j)$
and $y(j)$ such that the balls $B (x(j), s_j)$ and $B(y(j), s_j)$ are subsets of the first quadrant and 
$B (x(j), 3s_j) \cap B(y(j), 3s_j)= \emptyset$. 
We define a sequence of piecewise linear continuous functions $(v_j)_{j=1}^{\infty}$ by setting
\begin{equation*}
v_{j}(x):=
\begin{cases}
s_j^{- \frac{n-p}{p}}& \textrm{in } B (x_j^1, s_j), \\
-s_j^{- \frac{n-p}{p}} & \textrm{in } B (x_j^2, s_j),\\
0 & \textrm{in } G\setminus \left(B (x_j^1, 2s_j) \cup B (x_j^2, 2s_j) \right).
\end{cases}
\end{equation*}
Now we have
\[
\int_G |\nabla u_j|^p \, dx \le C s_j^n \left | \frac{s_j^{- \frac{n-p}{p}}}{s_j}\right|^p \le C
\]
for some constant $C$.
On the other hand, for any $b\in\R$
\[
\begin{split}
\int_{G} H(|u_j(x)-b|) \, dx &\ge C s_j^n H(|s_j^{- \frac{n-p}{p}} -b|) + C s_j^n H(|-s_j^{- \frac{n-p}{p}} -b|)\\ 
&\ge C s_j^n H(|s_j^{- \frac{n-p}{p}}|).
\end{split}
\]
Thus, there does not exist a
positive constant $C_1$ such that the inequality $\inf_b \|u-b \|_{L^H(G)} \le C_1 \| \nabla u\|_{L^p(G)}$
could hold for all $u$ from the appropriate space if
\[
\lim_{s \to \infty} s^n H(s^{- \frac{n-p}{p}}) = \lim_{s \to \infty} s^{\frac{pn}{n-p} }H\left(\frac1s \right)  =  \infty.
\]

By choosing  $H(t)= t^q$, we obtain that the inequality does not hold if 
\begin{equation}\label{equ:ehto2}
q < \frac{np}{n-p}.
\end{equation}

If $\lim_{t \to 0^+} t/\phi(t) = \infty$ and 
if there were an embedding with the Lebesgue space $L^q$ as a target space, then by  \eqref{equ:ehto1} we 
would have  $q  < \frac{np}{n-p}$ and by  \eqref{equ:ehto2} we would have  $q  \ge \frac{np}{n-p}$. Thus the target space cannot be a Lebesgue space.
The target space can be $L^q$ only if $\lim_{t \to 0^+} t/\phi(t) < \infty$. 
And in this case  $q= \frac{np}{n-p}$. Note that  the limit $\lim_{t \to 0^+} t/\phi(t)$ exists since $\phi$ is increasing and $\phi \ge 0$.  If  $\lim_{t \to 0^+} t/\phi(t)=m >0$, then there exists $t_0 >0$ such that $ \frac12 m \phi(t) \le t \le 2m \phi(t)$.

Thus, we have proved the following theorems.

\begin{theorem}
Let $\phi$ satisfy (1)--(5),  and assume that  $\lim_{t \to 0^+} t/\phi(t) =\infty$. Let $G$ be the unbounded $\phi$-cigar John domain constructed in \eqref{eq:mushroon-domain}. Let $1\le p < n$. Then there do not exist  numbers $q\in \R$  and $C\in\R$ such that the inequality
\[
\inf_{b\in \R} \| u- b \|_{L^q (G)} \le C \| \nabla u\|_{L^p(G)} 
\]
could hold for all $u \in L^1_p(G)$.
\end{theorem}

\begin{theorem}
Let $\phi$ satisfy (1)--(5),  and assume that  $\lim_{t \to 0^+} t/\phi(t) = m < \infty$. Let $G$ be the unbounded $\phi$-cigar John domain constructed in \eqref{eq:mushroon-domain}. Assume that there exist  numbers $q\in \R$  and $C\in\R$ such that the inequality
\[
\inf_{b\in \R} \| u- b \|_{L^q (G)} \le C \| \nabla u\|_{L^p(G)} 
\]
 holds for all $u \in L^1_p(G)$. Then $q= \frac{np}{n-p}$ and there exists $t_0>0$ such that $\phi(t) \approx t$ for all $t \in(0, t_0]$.
\end{theorem}

%%%%%%%%%%%%%%%%%%%%%%%%%%%%%%%%%%%%%%%%%%%%%%%%%%%%%%

%%%%%%%%%%%%%%%%%%%%%%%%%%%%%%%%%%%%%%%%%%%%%%%%%%%%%%%%%%%%%%%%%%%%

\bibliographystyle{amsalpha}

\end{document}